\documentclass[11pt]{article}

\usepackage[latin1]{inputenc}
\usepackage{t1enc}
\usepackage{amsmath, amssymb, amsfonts, amsthm, amsopn,epsfig}
\usepackage[none]{hyphenat}
\usepackage{tikz}
\usepackage{float}
\usepackage{graphicx}
\usepackage{caption}
\usepackage[margin=1in]{geometry}
\usepackage[toc,page]{appendix}
\usepackage{epstopdf}
\usepackage{etoolbox}
\usepackage[colorlinks=true]{hyperref}
\usepackage{dsfont}
\hypersetup{
	colorlinks=true,
	linkcolor=blue,
	filecolor=magenta,      
	urlcolor=cyan,
	citecolor=blue
}
\usepackage{cleveref}

\theoremstyle{plain}
\newtheorem{theorem}{Theorem}[section]

\newtheorem{claim}[theorem]{Claim}
\newtheorem{lemma}[theorem]{Lemma}

\theoremstyle{definition}

\title{Lower bounds for piercing and coloring boxes}
\author{Istv\'an Tomon\thanks{Ume\r{a} University, \emph{e-mail}: \textbf{istvan.tomon@umu.se}}}
\date{}

\begin{document}
	\sloppy 
	
	\maketitle

\begin{abstract}
    Given a family $\mathcal{B}$ of axis-parallel boxes in $\mathbb{R}^d$, let $\tau$ denote its \emph{piercing number}, and $\nu$ its \emph{independence number}. It is an old question whether $\tau/\nu$ can be arbitrarily large for given $d\geq 2$. Here, for every $\nu$, we construct a family of axis-parallel boxes achieving  $$\tau\geq \Omega_d(\nu)\cdot\left(\frac{\log \nu}{\log\log \nu}\right)^{d-2}.$$
    This not only answers the previous question for every $d\geq 3$ positively, but also matches the best known upper bound up to double-logarithmic factors.
    
    Our main construction has further implications about the Ramsey and coloring properties of configurations of boxes as well.  We show the existence of a family of $n$ boxes in $\mathbb{R}^{d}$, whose intersection graph has clique and independence number 
       $O_d(n^{1/2})\cdot \left(\frac{\log n}{\log\log n}\right)^{-(d-2)/2}.$ This is the first improvement over the trivial upper bound $O_d(n^{1/2})$, and matches the best known lower bound up to double-logarithmic factors. Finally,  for every $\omega$ satisfying $\frac{\log n}{\log\log n}\ll \omega\ll n^{1-\varepsilon}$, we construct an intersection graph of $n$ boxes with clique number at most $\omega$, and chromatic number  $\Omega_{d,\varepsilon}(\omega)\cdot \left(\frac{\log n}{\log\log n}\right)^{d-2}.$ This matches the best known upper bound up to a factor of $O_d((\log w)(\log \log n)^{d-2})$.
\end{abstract}

\section{Introduction}

Configurations of axis-parallel boxes in $\mathbb{R}^d$ are extensively studied in combinatorial and computational geometry. There is group of old problems addressing the relationship between the piercing number, clique number, independence number, and chromatic number of families of boxes (here and later, boxes are always axis-parallel, and closed, unless stated otherwise), where the best known lower and upper bounds match up to polylogarithmic factors. However, it remained an elusive problem to deal with these logarithms, mainly because there were no known constructions beating the trivial ones by more than an absolute constant factor. In this paper, for every $d\geq 3$, we provide such constructions, also reducing the gaps between the lower and upper bounds up to a factor of $(\log\log .)^{O(1)}$. Here and later, the $O(.),\Omega(.),\Theta(.)$ notations hide a constant which might depend on the dimension $d$, but no other parameter.

\subsection{Piercing number}

Given a family of sets $\mathcal{F}$, a \emph{piercing set} (or \emph{hitting set}) of $\mathcal{F}$ is a subset of $\bigcup_{F\in\mathcal{F}}F$ having a nonempty intersection with every element of $\mathcal{F}$. The \emph{piercing number} of $\mathcal{F}$, denoted by $\tau(\mathcal{F})$, is the minimal size of a piercing set. The \emph{independence number} (also known as \emph{matching number} or \emph{packing number}) of $\mathcal{F}$, denoted by $\nu(F)$, is the maximum number of pairwise disjoint elements of $\mathcal{F}$. Clearly, $\tau(\mathcal{F})\geq \nu(\mathcal{F})$, and it is a well known result of Gallai (see e.g. \cite{HS58}) that equality holds if $\mathcal{F}$ is a family of intervals.  The situation becomes more complicated if we move to higher dimensions. In 1965, Wegner \cite{W65} conjectured that if $\mathcal{R}$ is a family of rectangles, then $\tau(\mathcal{R})\leq 2\nu(\mathcal{R})-1$, while Gy\'arf\'as and Lehel \cite{GyL85} proposed the weaker conjecture that $\tau(\mathcal{R})\leq O(\nu(\mathcal{R}))$. To begin with, it is not even clear that $\tau(\mathcal{R})$ can be bounded by a function of $\nu(\mathcal{R})$ alone, but this follows from a result of K\'arolyi \cite{K91}, who proved that $\tau(\mathcal{R})\leq O(\nu(\mathcal{R})\log \nu(\mathcal{R}))$. This was improved by Correa, Feuilloley, P\'erez-Lantero, and Soto \cite{CFPS15} to the currently best known upper bound $\tau(\mathcal{R})\leq O(\nu(\mathcal{R}))\cdot(\log\log \nu(\mathcal{R}))^{2}$. In the special case when the elements of $\mathcal{R}$ are $r$-bounded (i.e. the ratio of the two sides of every rectangle in $\mathcal{R}$ is bounded by $r$), this upper bound can be improved to linear, see \cite{C20,CH12,CSZ18}.  From below, Jel\'inek (see \cite{CFPS15}) showed that $\tau(\mathcal{R})=2\nu(\mathcal{R})-4$ can be achieved for every $\nu\geq 4$, while $\nu(\mathcal{R})=4$ and $\tau(\mathcal{R})=7$ is also realizable \cite{CD20}.

Now consider $d\geq 3$.  Let $\mathcal{B}$ be a family of boxes in $\mathbb{R}^d$, and write $\tau=\tau(\mathcal{B})$ and $\nu=\nu(\mathcal{B})$. It is a simple result that  $\nu=1$ implies $\tau=1$, which is equivalent to the statement that a family of pairwise intersecting boxes has a nonempty intersection. In the next case $\nu=2$, Fon-der-Flaass and Kostochka \cite{FK93} showed that $\tau\leq d+1$, and that there exists a construction achieving $\tau=\Omega(d^{1/2}/\log d)$. For general $\nu$, Gy\'arf\'as and Lehel \cite{GyL85} observed that $\tau\leq \nu^{d}$ is easy to prove. This was subsequently improved by  K\'arolyi \cite{K91} to $\tau\leq O(\nu)\cdot(\log \nu)^{d-1}$, which was then reproved a number of times \cite{A07,CSZ18,N00}.  K\'arolyi's proof relies on a divide-and-conquer argument, and any improvement for the $d=2$ case translates into an improvement for higher dimensions as well. Therefore, the aforementioned result of \cite{CFPS15} implies that $\tau\leq O(\nu)\cdot(\log \nu)^{d-2}(\log\log \nu)^{2}$ also holds. For completeness, we present a proof of this in Section \ref{section:upper}, see Theorem \ref{thm:piercing_upper}. For certain special families of boxes, this upper bound can be improved \cite{CSZ18}: if $\mathcal{B}$ consists of \emph{cubes}, then $\tau=O(\nu)$; also, in case  for every two intersecting boxes, one contains a corner of the other, then $\tau=O(\nu\log\log \nu)$. On the other hand, it was a long-standing open problem whether $\tau/\nu$ can be arbitrarily large for any fixed $d$, see e.g. the survey of Eckhoff \cite{E03} (p. 359). Here, we show that the answer is yes for $d\geq 3$, and even more, our construction matches the best known upper bound up to a factor of $O((\log\log \nu)^{d})$.

 \begin{theorem}\label{thm:piercing}
Let $d\geq 3$, then there exists $c> 0$ such that the following holds. For every positive integer $k$, there exists a family of boxes in $\mathbb{R}^d$ containing at most $k$ pairwise disjoint sets and having piercing number at least
$$ck\left(\frac{\log k}{\log\log k}\right)^{d-2}.$$
\end{theorem}

\subsection{Ramsey number}

The \emph{intersection graph} of a family of sets $\mathcal{F}$ is the graph, whose vertices are the elements of $\mathcal{F}$, and two vertices are joined by an edge if they have a nonempty intersection. Ramsey properties of intersection graphs of geometric objects are extensively studied, see e.g. \cite{T20} as a general reference. Larman, Matou\v sek, Pach, and T\"or\H{o}csik \cite{LMPT94} proved that if $G$ is the intersection graph of $n$ boxes in $\mathbb{R}^{d}$, then $G$ contains either a clique or an independent set of size $\Omega(n^{1/2})\cdot(\log n)^{-(d-1)/2}$. Note that this bound can be improved to $\Omega(n^{1/2})\cdot(\log n)^{-(d-2)/2}(\log\log n)^{-1}$ using the above mentioned result about piercing numbers. Indeed, let $\mathcal{B}$ be a family of boxes realizing $G$. Setting $k=n^{1/2}(\log n)^{-(d-2)/2}(\log\log n)^{-1}$, either $\nu(\mathcal{B})\geq k$, which is equivalent to $\alpha(G)\geq k$, or  $\tau(G)\leq O(k)\cdot(\log k)^{d-2}(\log\log k)^{2}$. But then there is a point piercing at least $n/\tau(G)=\Omega(k)$ elements of $\mathcal{B}$, corresponding to a clique of size $\Omega(k)$.

From above, it is easy to construct a family of $n$ boxes, whose intersection graph contains no clique or independent set of size larger than $\sqrt{n}$. Indeed, the disjoint union of $\sqrt{n}$ cliques of size $\sqrt{n}$ can be realized as the intersection graph of boxes in any dimension. However, there were no known constructions improving this simple bound by more than a constant factor. Our next main result provides such a construction, and matches the previously described lower bound up to a factor of $O((\log \log n)^{d/2})$.

\begin{theorem}\label{thm:Ramsey}
Let $d\geq 3$, then there exists $c>0$ such that the following holds. For every positive integer $n$, there exists a family of $n$ boxes in $\mathbb{R}^{d}$, whose intersection graph $G$ contains no clique or independent set of size larger than $$cn^{1/2}\left(\frac{\log n}{\log\log n}\right)^{-(d-2)/2}.$$
\end{theorem}

\subsection{Coloring number}
As usual, $\omega(G),\chi(G),\alpha(G)$ denote the \emph{clique number}, \emph{chromatic number}, and \emph{independence number} of a graph $G$, respectively. A classical result of Asplund and Gr\"unbaum \cite{AG60} states that if $G$ is the intersection graph of rectangles, then $\chi(G)\leq O(\omega(G)^2)$. After sixty years, this was only recently improved by Chalermsook and Walczak \cite{CW21}  to $\chi(G)\leq O(\omega(G)\log \omega(G))$. On the other hand, there is no known construction beating the trivial bound $\chi(G)\geq \Omega(\omega(G))$.
 
 For $d\geq 3$, a celebrated construction of Burling \cite{B65} shows that the chromatic number of intersection graphs of boxes in $\mathbb{R}^d$ cannot be bounded by the clique number alone. The so called \emph{Burling graphs} provide an infinite family of graphs $G$ that are triangle-free, and if $G$ has $n$ vertices, then $\chi(G)=\Theta(\log\log n)$.  See \cite{RA08} for an alternative construction of a family of $K_5$-free intersection graphs of boxes in $\mathbb{R}^{3}$, whose chromatic number achieves a similar growth. Let $G$ be an intersection graph of $n$ boxes in $\mathbb{R}^d$ and let $\omega=\omega(G)$. Just as before, the best known upper bound on the chromatic number of $G$ is supplied by a divide-and-conquer argument (folklore). Using the aforementioned result of \cite{CW21} as base case, we get $\chi(G)\leq O(\omega)\cdot (\log\omega)(\log n)^{d-2}$. For completeness, we present the proof of this in Section \ref{section:upper}, see Theorem \ref{thm:chi_upper}. Here, we provide the first construction which matches this bound up to a factor of $O((\log w)(\log\log n)^{d-2})$ for a wide range of values of $\omega$.

 \begin{theorem}\label{thm:coloring}
Let $d\geq 3$ and $\varepsilon>0$, then there exists $c,c'\geq 0$ such that the following holds. For every positive integer $n$ and $\omega$ satisfying $c\frac{\log n}{\log\log n}\leq \omega\leq n^{1-\varepsilon}$, there exists a family of $n$ boxes in $\mathbb{R}^d$ such that its intersection graph $G$ satisfies $\omega(G)\leq \omega$ and 
$$\chi(G)\geq c'\cdot\omega\left(\frac{\log n}{\log\log n}\right)^{d-2}.$$ 
\end{theorem}

Unfortunately, Theorem \ref{thm:coloring} does not provide lower bounds in case $\omega(G)$ is a constant, the most interesting of which is when $\omega(G)=2$. It remains open whether there exists triangle-free intersection graphs of $n$ boxes in $\mathbb{R}^{d}$ of chromatic number growing much faster than $\log\log n$.

\section{The main theorem}
In this section, we state our main technical result, and show how to deduce Theorems \ref{thm:piercing}, \ref{thm:Ramsey}, and \ref{thm:coloring} from it. We omit the use of floors and ceilings whenever they are not crucial. 

\begin{theorem}\label{thm:main}
Let $d\geq 3$, then there exists $c_1,c_2\geq 0$ such that the following holds. For every positive integer $n$, there exists a family of $n$ boxes in $\mathbb{R}^d$ such that its intersection graph has clique number at most $c_1\frac{\log n}{\log\log n}$ and independence number at most $c_2 n (\frac{\log n}{\log\log n})^{-(d-1)}$
\end{theorem}


\begin{proof}[Proof of Theorem \ref{thm:piercing}]
    We may assume that $k$ is sufficiently large with respect to $d$. Let $c_1,c_2$ be the constants provided by Theorem \ref{thm:main}, and let $n$ be the largest positive integer such that $k\geq c_2n(\frac{\log n}{\log\log n})^{-(d-1)}$. Then $n=\Theta(k(\frac{\log k}{\log\log k})^{d-1})$. Let $\mathcal{B}$ be a family of $n$ boxes in $\mathbb{R}^{d}$ whose intersection graph $G$ satisfies $\omega(G)\leq c_1\frac{\log n}{\log\log n}$ and $\alpha(G)\leq c_2n(\frac{\log n}{\log\log n})^{-(d-1)}\leq k$. Note that a family of boxes that can be pierced by a single point is a clique in $G$, so the piercing number of $\mathcal{B}$ is at least $$\frac{n}{\omega(G)}\geq \frac{n\log\log n}{c_1\log n}\geq \Omega\left(k\left(\frac{\log k}{\log\log k}\right)^{d-2}\right).$$ Hence, $\mathcal{B}$ satisfies the required conditions.
\end{proof}

\begin{proof}[Proof of Theorem \ref{thm:Ramsey}]
We may assume that $n$ is sufficiently large with respect to $d$, and let $c_1,c_2$ be the constants provided by Theorem \ref{thm:main}. Let $n_0=n^{1/2}(\frac{\log n}{\log\log n})^{d/2}$ and $s=\frac{n}{n_0}=n^{1/2}(\frac{\log n}{\log\log n})^{-d/2}$. Let $\mathcal{B}_0$ be a family of $n_0$ boxes in $\mathbb{R}^{d}$, whose intersection graph $G_0$ satisfies $\omega(G_0)\leq c_1\frac{\log n_0}{\log\log n_0}$ and $\alpha(G_0)\leq c_2 n_0 (\frac{\log n_0}{\log\log n_0})^{-(d-1)}$. Let $\mathcal{B}$ be the family of $n=sn_0$ boxes we get after taking each element of $\mathcal{B}$ with multiplicity $s$, and let $G$ be the intersection graph of $\mathcal{B}$. Then $v(G)=n$, 
$$\omega(G)=s\cdot \omega(G_0)=O\left(n^{1/2}\left(\frac{\log n}{\log\log n}\right)^{-(d-2)/2}\right)$$
and 
$$\alpha(G)=\alpha(G_0)=O\left(n^{1/2}\left(\frac{\log n}{\log\log n}\right)^{-(d-2)/2}\right),$$
finishing the proof.
\end{proof}

\begin{proof}[Proof of Theorem \ref{thm:coloring}] We may assume that $n$ is sufficiently large with respect to $d$ and $\varepsilon$, and let $c_1,c_2$ be the constants provided by Theorem \ref{thm:main}. Let $c=c_1$, $s=\frac{\omega\log\log n}{c_1\log n}$ (so then $s\geq 1$), and $n_0=\frac{n}{s}\geq n^{\varepsilon-o(1)}$.  Let $\mathcal{B}_0$ be a family of $n_0$ boxes in $\mathbb{R}^{d}$, whose intersection graph $G_0$ satisfies $\omega(G_0)\leq c_1\frac{\log n_0}{\log\log n_0}$ and 
$$\alpha(G_0)\leq c_2 n_0 \left(\frac{\log n_0}{\log\log n_0}\right)^{-(d-1)}\leq c_3 \frac{n}{s}\cdot \left(\frac{\log n}{\log\log n}\right)^{-(d-1)},$$
where $c_3>0$ only depends on $c_2$ and $\varepsilon$. Let $\mathcal{B}$ be the family of $n=sn_0$ boxes we get after taking each element of $\mathcal{B}$ with multiplicity $s$, and let $G$ be the intersection graph of $\mathcal{B}$. Then $\omega(G)=s\cdot \omega(G_0)\leq \omega$ and $\alpha(G)=\alpha(G_0)$. On the other hand, we have
$$\chi(G)\geq \frac{n}{\alpha(G)}\geq \frac{s}{c_3} \left(\frac{\log n}{\log\log n}\right)^{d-1}=\frac{w}{c_1c_3}\cdot \left(\frac{\log n}{\log\log n}\right)^{d-2},$$
so $c'=\frac{1}{c_1c_3}$ suffices.
\end{proof}

\section{The construction}

This section is devoted to the proof of Theorem \ref{thm:main}. We fix the following parameters, some of which will be specified later with respect to $n$: let $s$ and $k$ be positive integers, and let $m:=s^k$ and $M:=m^{d-1}$. We assume that $s$ and $k$ are sufficiently large with respect to $d$. Given $\textbf{t}\in \mathbb{N}^d$ and $\textbf{p}\in \mathbb{Z}^d$, let $B_{\textbf{t}}(\textbf{p})$  denote the $s^{\textbf{t}(1)}\times\dots\times s^{\textbf{t}(d)}$ sized (not closed) box  $$\prod_{i=1}^{d}\left[s^{\textbf{t}(i)}\cdot \mathbf{p}(i),s^{\textbf{t}(i)}\cdot (\mathbf{p}(i)+1)\right).$$ Call $B_{\textbf{t}}(\textbf{p})$ a \emph{$\textbf{t}$-block}, or simply a \emph{block}. Clearly, the \emph{$\textbf{t}$-blocks} partition $\mathbb{R}^d$ for every $\textbf{t}\in \mathbb{N}^d$. Furthermore, an important observation is that for any $i\in [d]$, the projections of any two blocks onto the $i$-th axis are either disjoint, or one contains the other. In the special case $s=2$ and $d=2$, the  blocks we defined are sometimes referred to as \emph{dyadic} or  \emph{canonical rectangles}, and they play a key role in the celebrated result of Pach and Tardos \cite{PT12} about epsilon-nets, for example. For us, it will be crucial that $s$ can grow slowly with~$n$.

Let $$T=\left\{\textbf{t}\in \mathbb{N}^d: \sum_{i=1}^d \mathbf{t}(i)=k\right\},$$
and let $\mathcal{B}$ be the family of all blocks of volume $m=s^{k}$ contained in $[0,m]^d$. In other words, $\mathcal{B}$ is the family of all blocks $B_{\textbf{t}}(\textbf{p})$, where $\textbf{t}\in T$ and $\textbf{p}(i)\in \{0,\dots,s^{k-\mathbf{t}(i)}-1\}$ for every $i\in [d]$. Note that $|T|=\binom{k+d-1}{d-1}$ and $|\mathcal{B}|=|T|\cdot m^{d-1}=|T|\cdot M$. 

Let $G$ be the intersection graph of $\mathcal{B}$. While blocks are not closed boxes, we can slightly shrink them  and take their closures without changing the intersection pattern, so $G$ is also an intersection graph of closed boxes. Clearly, we have $\omega(G)=|T|$, as each point of $[0,m]^d$ is contained in a unique $\mathbf{t}$-block for every $\mathbf{t}\in T$. Also, $\alpha(G)=M$, as the $\mathbf{t}$-blocks form an independent set of size $M$ for every  $\mathbf{t}\in T$, and any family of $M+1$ blocks in $\mathcal{B}$ have total volume more than $m^{d}$. 

Our aim is to show that a random sample of the vertices of $G$ with some appropriate probability $p$ induces a subgraph $H$ satisfying the requirements of Theorem \ref{thm:main} with high probability. The clique number of a random sample is easy to analyze, as a clique is a set of blocks containing a given point. On the other hand, independent sets are harder to control, and we will use the celebrated \emph{graph container method} to do so. This method was first introduced by Kleitman and Winston \cite{KW82}, and later formalized by Sapozhenko \cite{S05}. In order to make our paper self contained, we do not assume familiarity with these publications, and the container method in general. Roughly, we show that there is a small collection $\mathcal{C}$ of subsets of $V(G)$, each of size close to $\alpha(G)$, such that every independent set of $G$ is covered by an element of $\mathcal{C}$. The members of the collection $\mathcal{C}$ are called \emph{containers}. Such a collection exists if $G$ admits a \emph{supersaturation} result, i.e. every subset of the vertices slightly larger than $\alpha(G)$ induces a subgraph of large maximum degree. We prove this supersaturation result in the next lemma.

\begin{lemma}\label{lemma:saturation}
 Let  $\varepsilon>0$, and let $S\subset \mathcal{B}$ be of size at least $(1+\varepsilon)M$. Then the maximum degree of $G[S]$ is at least $\frac{\varepsilon s}{|T|^2}$. 
\end{lemma}

\begin{proof}
 Let $\Delta$ be the maximum degree of $G[S]$. For a block $C$ and $\mathbf{t}\in T$, let $S_{\mathbf{t}}[C]$ denote the set of $\mathbf{t}$-blocks among the elements of $S$  that are contained in $C$.
 
 Fix a pair of distinct $\textbf{t},\textbf{u}\in T$, and define $\mathbf{w}\in \mathbb{N}^d$ such that $\mathbf{w}(i)=\max\{\mathbf{t}(i),\mathbf{u}(i)\}$.  Observe that if $B$ is a $\mathbf{t}$-block and $B'$ a is $\mathbf{u}$-block, then  $B$ and $B'$ have a nonempty intersection if and only if  $B$ and $B'$ are contained in the same $\mathbf{w}$-block. Indeed, for $i\in [d]$ let $B_i$ and $B_i'$ be the projections of $B$ and $B'$ onto the $i$-th axis, respectively. Then $B$ and $B'$ intersect if and only if $B_i\subset B_i'$ or $B_i'\subset B_i$ for every $i\in [d]$ (indeed, recall that if $B_i\cap B_i'\neq \emptyset$, then one must contain the other). Thus, if $B\cap B'\neq \emptyset$, then $(B_1\cup B_1')\times \dots \times (B_d\cup B_d')$ is a $\mathbf{w}$-block containing $B$ and $B'$. On the other hand, if $C$ is a $\mathbf{w}$-block containing $B$ and $B'$, and $C_i$ is the projection of $C$ onto the $i$-th axis, then either $C_i=B_i$ or $C_i=B_i'$, showing that $B_i\subset B'_i$ or $B'_i\subset B_i$.
 
 Define the family $\mathcal{M}(\textbf{t},\textbf{u})\subset S$ as follows. By the previous, for every $\mathbf{w}$-block $C$,  every element of $S_{\mathbf{t}}[C]$ intersects every element of $S_{\mathbf{u}}[C]$. Therefore, we have either $|S_{\mathbf{t}}[C]|\leq \Delta$ or $|S_{\mathbf{u}}[C]|\leq \Delta$. In the first case, we add every element of $S_{\mathbf{t}}[C]$ to $\mathcal{M}(\textbf{t},\textbf{u})$, in the second case we add every element of $S_{\mathbf{u}}[C]$ to $\mathcal{M}(\textbf{t},\textbf{u})$ (if both cases hold, we only add the elements of $S_{\mathbf{t}}[C]$, say). The total number of $\mathbf{w}$-blocks in $[0,m]^d$ is $$\frac{m^d}{s^{\mathbf{w}(1)+\dots+\mathbf{w}(d)}}\leq \frac{m^{d}}{s^{k+1}}=\frac{M}{s},$$ where the first inequality holds by noting that $\mathbf{w}(1)+\dots+\mathbf{w}(d)>\mathbf{t}(1)+\dots+\mathbf{t}(d)=k$. Hence, $|\mathcal{M}(\textbf{t},\textbf{u})|\leq \Delta\cdot\frac{M}{s}$. Moreover, $S\setminus \mathcal{M}(\textbf{t},\textbf{u})$ contains no $\textbf{t}$-block intersecting a $\textbf{u}$-block, because if a $\mathbf{t}$-block intersects a $\mathbf{u}$-block, then they are contained in the same $\mathbf{w}$-block.

 Let $S'=S\setminus\bigcup_{\mathbf{t},\mathbf{u}\in T,\mathbf{t}\neq\mathbf{u}}\mathcal{M}(\mathbf{t},\mathbf{u})$. Then $|S'|\geq |S|-|T|^2 \Delta \frac{M}{s}$. Crucially, any two blocks in $S'$ are disjoint, so $|S'|\leq M$. Therefore, comparing the lower and upper bound on $|S'|$, we arrive to the inequality $$M\geq |S'|\geq |S|-|T|^2 \Delta \frac{M}{s} \geq (1+\varepsilon)M- |T|^2 \Delta \frac{M}{s}.$$
 Comparing the left and right sides, we get the desired inequality $\Delta\geq \frac{\varepsilon s}{|T|^2}$.
\end{proof}

With a slightly more involved calculation, the lower bound on the maximum degree can be improved to $\Omega(\frac{\epsilon s}{|T|})$. However, this improvement has no effect on the bound in Theorem \ref{thm:main}. Moreover, it is worth noting that in case $d=2$, one has the much stronger saturation result that any set of size $M+1$ induces a subgraph of maximum degree $\Omega(s)$. This no longer holds for $d\geq 3$, there exists a set of $M+1$ boxes that induces a subgraph of maximum degree 1. We leave these claims as exercises, and they will not be used later.

Now we are ready to state our container lemma, whose proof should be mostly standard for anyone familiar with the container method.

\begin{lemma}\label{lemma:containers}
    Let $s\geq |T|^3$. Then there exists a collection $\mathcal{C}$ of subsets of $\mathcal{B}$ such that
    \begin{enumerate}
        \item[(1)] every $C\in\mathcal{C}$ satisfies $|C|\leq 3M$,
        \item[(2)] $|\mathcal{C}|\leq e^{(\log s) M|T|^3/s}$,
        \item[(3)] every independent set of $G$ is contained in some element of $\mathcal{C}$.
    \end{enumerate}
\end{lemma}

\begin{proof}
Let $<$ be an arbitrary total ordering of the elements of $\mathcal{B}$. For a graph $H$ and vertex $v\in V(H)$, $N_{H}(v)=\{w\in V(H):vw\in E(H)\}$ denotes the \emph{neighborhood} of $v$ in $H$.

Fix an independent set $I$ of $G$. We construct a \emph{fingerprint} $S$ and a set $f(S)$ (depending only on $S$) for $I$ with the help of the following algorithm.

Let $S_0=\emptyset$ and $G_0=G$. If $S_i$ and $G_i$ are already defined, we define $S_{i+1}$ and $G_{i+1}$ in the following manner. Let $v$ be the first vertex (with respect to $<$) of maximum degree in $G_{i}$.  
\begin{itemize}
\item If $|V(G_{i})|\leq 2M$, then stop, and set $S=S_{i}$ and $f(S)=V(G_{i})$.
\item  Otherwise, if $v\not\in I$, then set $S_{i+1}=S_i$, remove $v$ from $G_i$, and let the resulting graph be $G_{i+1}$. 
\item If $v\in I$, then set $S_{i+1}=S_i\cup\{v\}$, and remove $\{v\}\cup N_{G_i}(v)$ from $G_i$, let $G_{i+1}$ be the resulting graph.
\end{itemize}
Let us analyze this algorithm. At each step, the size of $G_i$ decreases, so the algorithm stops after a finite number of steps. 

Firstly, let us argue that $f(S)$ indeed only depends on $S$. In particular, we show that if $I_1,I_2$ are independent sets, for which the algorithm produces the same fingerprint $S$, then the algorithm produces the same sets $S_i$ and graphs $G_i$ at every step  for $I_1$ and $I_2$. Otherwise, let $i$ be the last step, where $G_{j}$ and $S_{j}$ agree for $I_1$ and $I_2$ for every $j\leq i$. Then the first vertex $v$ of maximum degree in $G_i$ is contained in exactly one of $I_1$ and $I_2$, say $v\in I_1\setminus I_2$. But then the algorithm for $I_1$ sets $S_{i+1}=S_i\cup\{v\}$, so $v\in S$. On the other hand, the algorithm for $I_2$ sets $S_{i+1}=S_i$, removes $v$ from $G_i$, and never processes this vertex again, so $v\not\in S$, contradiction.

Secondly, observe that as $I$ is an independent set, we have $I\subset S_i\cup V(G_i)$ for every $i$, so in particular $I\subset S\cup f(S)$. 

Finally, we have $|S|\leq \frac{M|T|^3}{s}$. Indeed, in case we added a vertex $v$ to $S_i$ to get $S_{i+1}$, we removed at least $1+|N_{G_i}(v)|$ vertices from $G_i$ to get $G_{i+1}$. But $v$ is a vertex of maximum degree in $G_i$, and $G_i$ has at least $2M$ vertices, so applying Lemma \ref{lemma:saturation} with $\varepsilon=1$, we get $1+|N_{G_i}(v)|> \frac{s}{|T|^2}$. Therefore, as we started with $|\mathcal{B}|$ vertices, there are at most $|\mathcal{B}|\frac{|T|^2}{s}=\frac{M|T|^3}{s}$ indices $i$ for which $|S_{i+1}|= |S_i|+1$. In particular, we have $|S\cup f(S)|\leq \frac{M|T|^3}{s}+2M\leq 3M$

Let $\mathcal{C}$ be the collection of all the set $S\cup f(S)$, where $S$ is the fingerprint of some independent set $I$. Then (1) and (3) are clearly satisfied. Also, as each fingerprint has size at most $\frac{M|T|^3}{s}$, we have 
$$|\mathcal{C}|\leq \binom{|\mathcal{B}|}{M|T|^3/s}\leq \left(\frac{es}{|T|^2}\right)^{M|T|^3/s}\leq e^{(\log s) M|T|^3/s}.$$
Here, the second inequality holds by the general inequality $\binom{a}{b}\leq (\frac{ea}{b})^{b}$ and substituting $|\mathcal{B}|=|T| M$. This finishes proof.
\end{proof}

Let $p\in (0,1)$ be a parameter specified later, and sample the elements of $\mathcal{B}$ independently with probability $p$. Let the resulting sample be $X$, and let $H=G[X]$. In the following claims, we collect the important properties of $H$. During our arguments, we use standard concentration arguments.

\begin{lemma}[Multiplicative Chernoff bound]\label{lemma:chernoff}
 Let $X$ be the sum of independent indicator random variables. If $\lambda\geq 2\mathbb{E}(X)$, then
 $\mathbb{P}(X\geq \lambda)\leq e^{-\lambda/6}$. Also, 
 $\mathbb{P}(X\leq \mathbb{E}(X)/2)\leq e^{-\mathbb{E}(X)/8}.$
\end{lemma}

First, we bound the size of $X$.

\begin{claim}\label{claim:size}
If $p\geq \frac{1}{s}$, then $|X|\geq \frac{p|\mathcal{B}|}{2}$  with probability at least $3/4$.
\end{claim}

\begin{proof}
Follows simply from the multiplicative Chernoff bound:
$$\mathbb{P}\left(|X|<\frac{p|\mathcal{B}|}{2}\right)\leq e^{-p|\mathcal{B}|/8}<\frac{1}{4},$$
where in the last inequality we used that $p|\mathcal{B}|>s^{kd-1}$ is sufficiently large.
\end{proof}

Next, we bound the independence number of $H$. We argue that as the number of containers covering the independent sets of  $G$ is small, each of them must shrink to approximately $p$ proportion after sampling. As the containers have size $O(\alpha(G))$, this ensures that $\alpha(H)=O(p\alpha(G))$.

\begin{claim}\label{claim:independent_set}
If $s\geq |T|^{3}$ and $p>\frac{2(\log s)|T|^3}{s}$, then $\alpha(H)\leq 6pM$ with probability at least $3/4$.
\end{claim}

\begin{proof}
    Let $\mathcal{C}$ be a collection of subsets of $\mathcal{B}$ satisfying the conditions of Lemma \ref{lemma:containers}. Let $C\in \mathcal{C}$, then $|C|\leq 3M$. Hence, $\mathbb{E}(|C\cap X|)\leq 3pM$. By the multiplicative Chernoff bound (Lemma \ref{lemma:chernoff}), we have 
    $$\mathbb{P}(|C\cap X|\geq 6pM)\leq e^{-pM}.$$
    Note that $$|\mathcal{C}|\cdot e^{-pM}\leq e^{(\log s) M|T|^3/s-pM}<\frac{1}{4}.$$ Hence, by the union bound, with probability at least $3/4$, we have $|C\cap X|\leq 6pM$ for every $C\in \mathcal{C}$. But every independent set of $H$ is contained in $C\cap X$ for some $C\in\mathcal{C}$, so we also have $\alpha(H)\leq 6pM$.
\end{proof}

Finally, let us bound the clique number of $H$.

\begin{claim}\label{claim:clique}
If $p\leq \frac{1}{e|T|}$, then $\omega(H)< \frac{2dk\log s}{\log k}$ with probability at least $3/4$.
\end{claim}

\begin{proof}
Let $r=\lceil \frac{2d(\log m)}{\log\log m}\rceil$ and let $\mathbf{o}\in\mathbb{N}^d$ be the all zero vector. Note that if $\mathcal{C}\subset \mathcal{B}$ is a clique, then the intersection of the elements of $\mathcal{C}$ contains an $\mathbf{o}$-block. Let $B$ be an $\mathbf{o}$-block of $[0,m]^d$, and let $\mathcal{C}_0\subset \mathcal{B}$ be the elements containing $B$. Then $|\mathcal{C}_0|=|T|$, as for every $\mathbf{t}\in T$, there is a unique $\mathbf{t}$-block containing $B$. Therefore, we have
$$\mathbb{P}(|\mathcal{C}_0\cap X|\geq r)\leq \binom{|T|}{r}p^r\leq \left(\frac{ep|T|}{r}\right)^{r}\leq e^{-r\log r}\leq e^{-2d\log m}\leq \frac{1}{4m^{d}}.$$
The number of $\textbf{o}$-blocks in $[0,m]^d$ is $m^d$, so the union bound shows that with probability at least $3/4$, no $\textbf{o}$-block is contained in $r$ elements of $X$. Thus, recalling that $m=s^{k}$, we have $$\omega(H)< \frac{2d\log m}{\log\log m}<\frac{2dk\log s}{\log k}$$ with probability at least $3/4$.
\end{proof}

Now everything is set to prove our main theorem.

\begin{proof}[Proof of Theorem \ref{thm:main}]
 Given $n$, we will choose the parameters $k,s,p$ such that the graph $H$ defined above has (almost) $n$ vertices, and the conditions of Claims \ref{claim:size}, \ref{claim:independent_set} and \ref{claim:clique} are satisfied. We may assume that $n$ is sufficiently large with respect to $d$, which will ensure that $k$ and $s$ are also sufficiently large with respect to $d$.

Let $k$ be the solution of the equation $n=\frac{1}{8}\cdot k^{4d(d-1)k}$, and let $s=k^{4d}$. Then $n=\frac{1}{8}\cdot s^{k(d-1)}=\frac{M}{8}$, $k=\Theta(\frac{\log n}{\log\log n})$ and 
$$ k^{d-1}\geq |T|=\Theta\left(\left(\frac{\log n}{\log\log n}\right)^{d-1}\right).$$ 
Set $p=\frac{1}{4|T|}$, then $p$ clearly satisfies the condition of Claim \ref{claim:clique}. Also, we have $s>k^{3d}>|T|^{3}$ and  
$$\frac{2(\log s)|T|^4}{s}\leq \frac{4d(\log k) k^{4(d-1)}}{k^{4d}}<\frac{1}{4},$$
which implies that $p$ satisfies the condition of Claim \ref{claim:independent_set}, and thus the condition of Claim \ref{claim:size} as well. Therefore, with probability at least $1/4$, the following are simultaneously  satisfied
\begin{itemize}
    \item $|X|\geq \frac{p|\mathcal{B}|}{2}=\frac{M}{8}=n$,
    \item  $\alpha(H)\leq 6pM=\frac{12n}{|T|}=O(n(\frac{\log n}{\log\log n})^{-(d-1)})$,
    \item $\omega(H)<\frac{2dk\log s}{\log k}=O(k)=O(\frac{\log n}{\log\log n})$.
\end{itemize}
After possibly removing some further vertices of $H$, we get an intersection graph of $n$ boxes in $\mathbb{R}^d$, satisfying the required conditions.
\end{proof}

We remark that in case $d=2$, the graph $G$ defined above is a comparability graph. Indeed, if $B$ is a $\mathbf{t}$-block and $B'$ is $\mathbf{t}'$-block, then writing $B\prec B'$ if $B\cap B'\neq\emptyset$ and $\mathbf{t}(1)<\mathbf{t}'(1)$, $\prec$ is a partial ordering, whose comparability graph is $G$. But then $\omega(H)=\chi(H)$ for every induced subgraph $H$ of $G$, showing that unfortunately one cannot get interesting constructions by considering subgraphs of~$G$. 

\section{Upper bounds}\label{section:upper}

As promised in the introduction, we present the upper bounds on the piercing number and chromatic number of boxes. Without loss of generality, we may assume that any family of boxes under consideration is in \emph{general position}, that is, the corners of the boxes have different coordinates.

In addition, we will use the following notation. Given a family of boxes $\mathcal{B}$ in $\mathbb{R}^d$ and $t\in\mathbb{R}$, let $\mathcal{B}^{-}(t)$ denote the set of elements of $\mathcal{B}$ that are contained entirely in the open half-space $\{\mathbf{x}\in\mathbb{R}^{d}:\mathbf{x}(d)<t\}$.  Similarly, let $\mathcal{B}^{+}(t)$ denote the set of elements of $\mathcal{B}$ that are contained entirely in the open half-space $\{\mathbf{x}\in\mathbb{R}^{d}:\mathbf{x}(d)>t\}$, and set $\mathcal{B}^{0}(t)=\mathcal{B}\setminus (\mathcal{B}^-(t)\cup \mathcal{B}^+(t))$. Finally, let $\mathcal{B}(t)=\{B\cap H:B\in \mathcal{B}\}$, where $H=\{\mathbf{x}\in\mathbb{R}^{d}:\mathbf{x}(d)=t\}$. Then $\mathcal{B}(t)$ is equivalent to a family of boxes in $\mathbb{R}^{d-1}$, and the intersection graph of $\mathcal{B}(t)$ is isomorphic to the intersection graph of $\mathcal{B}^{0}(t)$.

Furthermore, we use the following well known inequalities on recursively bounded functions. Let $f,g:\mathbb{Z}^{+}\rightarrow{R}^{+}$ be monotone increasing functions. If $f(n)\leq f(\lceil n/2\rceil)+g(n)$ holds for every $n\geq 2$, then $f(n)\leq cg(n)\log n$  for some $c>0$ depending only on $f(1)$. Also, if $g$ is convex and $f(n)\leq f(\lfloor n/2\rfloor)+f(\lceil n/2\rceil)+g(n)$ holds for every $n\geq 2$, then $f(n)\leq cg(n)\log n$ for some $c>0$ depending only on $f(1)$. 

\begin{theorem}\label{thm:piercing_upper}
    For every $d\geq 2$ there exists a constant $c_d$ such that if $\mathcal{B}$ is a family of boxes in $\mathbb{R}^d$ satisfying  $\nu(\mathcal{B})\leq k$, then $\tau(\mathcal{B})\leq c_d k(\log k)^{d-2}(\log \log k)^{2}$.
\end{theorem}

\begin{proof}
Let $f_d(k)$ denote the maximum possible piercing number of a family of boxes in $\mathbb{R}^d$ with independence number $k$, and let $\mathcal{B}$ be a family of boxes realizing it. It is easy to show that $f_{d}$ is a convex function by considering the disjoint union of two families of boxes.  We proceed by induction on $d$ and $k$ to show that $f_d(k)\leq c_d  k(\log k)^{d-2}(\log \log k)^{2}$ with some suitable constant $c_d$. This holds for $d=2$ and every $k\in \mathbb{Z}^+$ by \cite{CFPS15}. Also, we have $f_d(1)=1$ for every $d\in\mathbb{Z}^+$. Assume $d\geq 3$ and $k\geq 2$. It is easy to see that there exists $t$ such that $\nu(\mathcal{B}^-(t))=\lfloor k/2\rfloor$ by the assumption that the boxes are in general position. But then we must have $\nu(\mathcal{B}^+(t))\leq \lceil k/2\rceil$, as no member of $\mathcal{B}^-(t)$ intersects a member of $\mathcal{B}^+(t)$. Finally, we have $\nu(\mathcal{B}(t))=\nu(\mathcal{B}^{0}(t))\leq k$ and $\tau(\mathcal{B}^{0}(t))=\tau(\mathcal{B}(t))\leq f_{d-1}(k)$. From this, we get 
\begin{align*}
    \tau(\mathcal{B})&\leq \tau(\mathcal{B}^{-}(t))+\tau(\mathcal{B}^{+}(t))+\tau(\mathcal{B}^{0}(t))\\
    &\leq f_d(\lfloor k/2\rfloor)+f_d(\lceil k/2\rceil)+f_{d-1}(k)\\
    &\leq cf_{d-1}(k) \log k\leq c_d k(\log k)^{d-2}(\log\log k)^2,
\end{align*}
where $c$ is some absolute constant, and $c_{d}$ only depends on $c_{d-1}$.
\end{proof}

With slight abuse of notation, write $\chi(\mathcal{B})$ and $\omega(\mathcal{B})$ for $\chi(G)$ and $\omega(G)$, respectively, where $G$ is the intersection graph of the family of boxes $\mathcal{B}$. 

\begin{theorem}\label{thm:chi_upper}
    For every $d\geq 2$ there exists a constant $c_d$ such that if $\mathcal{B}$ is a family of $n$ boxes in $\mathbb{R}^d$ satisfying  $\omega(\mathcal{B})\leq \omega$, then $\chi(\mathcal{B})\leq c_d \omega (\log \omega)(\log n)^{d-2}$.
\end{theorem}

\begin{proof}
Let $f_{d,\omega}(n)$ denote the maximum chromatic number of a family of $n$ boxes in $\mathbb{R}^d$ with clique number at most $\omega$, and let $\mathcal{B}$ be family of $n$ boxes realizing it.   We proceed by induction on $d$ to show that $f_{d,\omega}(n)\leq c_d \omega (\log \omega)(\log n)^{d-2}$ with some suitable constant $c_d$. This holds for $d=2$ and every $n,\omega\in \mathbb{Z}^+$ by \cite{CW21}. Also, we have $f_{d,\omega}(1)=1$ for every $d\in\mathbb{Z}^+$. Assume $d\geq 3$ and $n\geq 2$. It is easy to see that there exists $t$ such that $|\mathcal{B}^-(t)|=\lfloor n/2\rfloor$ by the assumption that the boxes are in general position. But then  $|\mathcal{B}^+(t))|\leq \lceil n/2\rceil$. Moreover, we have $\omega(\mathcal{B}(t))=\omega(\mathcal{B}^{0}(t))\leq \omega$ and $\chi(\mathcal{B}^{0}(t))=\chi(\mathcal{B}(t))\leq f_{d-1,\omega}(n)$. Consider the coloring in which we use the same set of at most $f_{d,\omega}(\lceil n/2\rceil)$ colors to color both $\mathcal{B}^-(t)$ and $\mathcal{B}^+(t)$, and use a different set of at most $f_{d-1,\omega}(n)$ colors to color $\mathcal{B}^{0}(t)$. This shows that
\begin{align*}
    \chi(\mathcal{B})&\leq f_{d,\omega}(\lceil n/2\rceil)+f_{d-1,\omega}(n)\\
    &\leq cf_{d-1,\omega}(n) \log n\leq c_d \omega(\log \omega)(\log n)^{d-2},
\end{align*}
where $c$ is some absolute constant, and $c_{d}$ only depends on $c_{d-1}$.
\end{proof}

\section{Concluding remarks}

J\'anos Pach and G\'abor Tardos brought to my attention that there is more straightforward way of dealing with independent sets in the proof of Theorem \ref{thm:main}, which does not rely on the container method.

It follows form the proof of Lemma \ref{lemma:saturation} that the graph $G$ is the edge disjoint union of $q=O(\frac{M|T|^2}{s})$ complete bipartite graphs. Let $X_i,Y_i$ be the vertex classes of these bipartite graphs, where $i=1,\dots,q$. Define the collection  $\mathcal{I}$ of subsets of $V(G)$ as follows. For every $i\in [q]$, let $Z_i$ be equal to either $X_i$ or $Y_i$, and let $I=I(Z_1,\dots,Z_q)$ be the set of vertices $x\in V(G)$ such that whenever $x\in X_i\cup Y_i$, we have $x\in Z_i$. Let $\mathcal{I}$ be the collection of the $2^q$ such sets $I$, then $|\mathcal{I}|\leq 2^{O(M|T|^2/s)}$. Observe that every element of $\mathcal{I}$ is an independent set, and  $\mathcal{I}$ contains every \emph{maximal} independent set of $G$ (with respect to containment). Therefore, in the proof of Theorem \ref{thm:main}, instead of working with the collection of containers $\mathcal{C}$, one can work with $\mathcal{I}$ directly. However, this does not effect the quantitative bounds in Theorem \ref{thm:main}.

\bigskip

\noindent
\textbf{Acknowledgements.} We would like to thank J\'anos Pach and G\'abor Tardos for their valuable remarks.

\end{document}